\documentclass[a4paper,12pt]{article}
\usepackage{CJKutf8}
\usepackage{CJK}
\usepackage{amsfonts}
\usepackage{pifont}
\usepackage{bm}
\usepackage{latexsym,amsmath,amssymb,cite,amsthm}
\usepackage{color,eucal,enumerate,mathrsfs}
\usepackage[normalem]{ulem}
\usepackage{amsmath}
\usepackage{hyperref}

\numberwithin{equation}{section}
\theoremstyle{plain}
\newtheorem{theorem}{Theorem}[section]
\newtheorem{corollary}[theorem]{Corollary}

\newtheorem{proposition}[theorem]{Proposition}
\theoremstyle{definition}
\newtheorem{definition}[theorem]{Definition}
\theoremstyle{remark}

\newtheorem{remark}[theorem]{Remark}

\newcommand{\lip}[1]{{\mathrm{lip}}({#1})}

\newcommand{\Ric}{{\rm{Ricci}}}

\newcommand{\bRic}{{\bf Ricci}}

\renewcommand{\H}{{\mathrm{H}}}

\newcommand{\Hess}{{\mathrm{Hess}}}

\newcommand{\mm}{\mathfrak m}
\newcommand{\ms}{(X,\d,\mm)}

\newcommand{\rcdkn}{{\rm RCD}^*(K, N)}
\newcommand{\rcd}{{\rm RCD}(K, \infty)}

\newcommand{\R}{\mathbb{R}}

\renewcommand{\d}{{\mathrm d}}

\newcommand{\D}{{\mathrm D}}
\newcommand{\restr}[1]{\lower3pt\hbox{$|_{#1}$}}
\newcommand{\la}{{\langle}}
\newcommand{\ra}{{\rangle}}

\newcommand{\nchi}{{\raise.3ex\hbox{$\chi$}}}

\newcommand{\lims}{\varlimsup}

\title{\large{{\bf Ricci tensor on smooth metric measure space with boundary}}
}
\begin{document}
\author{Bang-Xian Han \thanks
{ University of Bonn,  Institute for Applied Mathematics, han@iam.uni-bonn.de}
}

\date{\today} 
\maketitle
\begin{abstract}
The aim of this note is to study the measure-valued Ricci tensor on smooth metric measure space with boundary, which is a generalization of Bakry-\'Emery's modified Ricci tensor on weighted Riemannian manifold. As an application, we offer a new approach to study curvature-dimension condition of smooth metric measure space with boundary.
\end{abstract}


\section{Introduction}
Let $M=(X, {\rm g}, e^{-V}{\rm Vol}_{\rm g})$ be a $n$-dimensional weighted Riemannian manifold (or smooth metric measure space) equipped with a metric tensor ${\rm g}: [TM]^2 \mapsto C^\infty(M)$.  The well-known Bakry-\'Emery's  Bochner  type formula
\begin{equation}\label{bf}
\Gamma_2(f)=\Ric(\nabla f, \nabla f)+\H_V(\nabla f, \nabla f)+| \H_f |_{\rm HS}^2,
\end{equation}
valid for any smooth function $f$, where $\H_V=\nabla^2 V$ is the Hessian of $V$ and $| \H_f |_{\rm HS}$ is the Hilbert-Schmidt norm of the Hessian $\H_f$. The operator $\Gamma_2$ is defined by
\[
\Gamma_2(f):=\frac12 L \Gamma(f,f) -\Gamma(f, Lf),\qquad\Gamma(f,f):=\frac12L(f^2)-fLf
\]
where $\Gamma(\cdot, \cdot)={\rm g}(\nabla \cdot, \nabla \cdot )$, and $L=\Delta-\nabla V$ is the Witten-Laplacian on $M$.  It is known that $\Gamma_2 \geq K$ could characterize many important geometric and analysis properties of $M$.

The aim of this paper is to study  the Bakry-\'Emery's $\Gamma_2$-calculus on smooth metric measure space with boundary. It can be seen that  smooth metric measure space with boundary is actually a non-smooth space,  since  the geodesics are not even $C^2$ in general (see e.g. \cite{AA-G}).  Therefore,  it will not be more difficult to study this problem in an abstract framework. In this paper,  we will use the theory of (non-smooth) metric measure space with lower Ricci curvature bound, which was founded by Lott-Sturm-Villani, and systematically studied using different techniques which originally come from differential geometry, metric geometry, probability theory, etc. 

We will see that the non-smooth Bochner inequality and the measure-valued Ricci tensor $\bRic$, which are introduced in \cite{S-S} and \cite{G-N} have  precise representations on weighted Riemannian manifold $(\Omega, \d_{\rm g}, e^{-V}{\rm Vol}_{\rm g})$ with boundary, where $\d_{\rm g}$ is the intrinsic distance on $\Omega\subset X$ induced by the Riemannian metric ${\rm g}$:
\begin{equation}\label{eq:intro-1}
\bRic_\Omega=\Ric_V\,e^{-V}\d {\rm Vol}_{\rm g}+II \,e^{-V}\d \mathcal{H}^{n-1}\restr{\partial \Omega}
\end{equation}
where $\Ric_V=\Ric+\H_V$ is the Bakry-\'Emery Ricci tensor and $II$ is the second fundamental form.

From  \cite{AGS-B, AGS-M} and \cite{G-N} we know that $(\Omega, \d_{\rm g}, {\rm Vol}_{\rm g})$ is a $\rcd$ space, or in other words, the Boltzman entropy is $K$-displacement convex, if and only if $ \bRic_\Omega \geq K$.   By \eqref{eq:intro-1} we know  $ \bRic_\Omega \geq K$ if and only if $\Ric \geq K$ and $II \geq 0$. Then we immediately know $(\Omega, \d_{\rm g})$ is locally convex if it is $\rcd$. Even though this result could also  be proved by combining  the result of Ambrosio-Gigli-Savar\'e (\cite{AGS-C, AGS-M}) and  Wang (see e.g. Chapter 3, \cite{WFY-A}). Our approach here is the first one totally `inside' the framework of metric measure space.

In this paper,  we will review the construction of measure-valued Ricci tensor and give a quick proof to our main formula. Then we end this note with some  simple applications. More applications and generalizations will be studied in the future.

\section{Measure valued Ricci tensor and application}
Let $M:=\ms$ be a complete, separable geodesic space.
We define the  local Lipschitz constant $\lip f:X\to[0,\infty]$  of a function $f$ by
\[
\lip f(x):=
\left\{ \begin{array}{ll}
\mathop{\lims}_{y\to x}\frac{|f(y)-f(x)|}{\d(x,y)},~~~~x~~\text{is not isolated}\\
0,~~~~~~~~~~~~~~~~~~~~~~~~\text{otherwise.}
\end{array}\right.
\]

We say that $f\in L^2(X, \mm)$ is a Sobolev function in $W^{1,2}(M)$ if there exists a sequence of Lipschitz functions functions $\{f_n\} \subset L^2$,  such that $f_n \to f$ and $\lip{f_n} \to G$ in $L^2$ for some $G \in L^2(X, \mm)$. It is known that there exists a minimal function $G$ in $\mm$-a.e. sense. We call this minimal $G$ the minimal weak  upper gradient (or weak gradient for simplicity) of the function $f$, and denote it by $|\D f|$.
It is known that the locality holds for $|\D f|$, i.e. $|\D f|=|\D g|$ $\mm$-a.e. on the set $\{x\in X:f(x)=g(x)\}$. If  $M$ is a Riemannian manifold, it is known that $|\D f|_M=|\nabla f|=\lip{f}$  for any $f\in C^\infty$. Furthermore, let  $\Omega \subset M$ be a domain such that  $\partial \Omega $ is  $(n-1)$-dimensional. Then we know $|\D f|_\Omega=|\nabla f|$ $\mm$-a.e. (see Theorem  6.1, \cite{C-D}). It can also be seen that the weighted measure $e^{-V}\mm$ does not change the value of weak gradients.

We equip $W^{1,2}\ms$ with the norm
\[
\|f\|^2_{W^{1,2}\ms}:=\|f\|^2_{L^2(X,\mm)}+\||\D f|\|^2_{L^2(X,\mm)}.
\]
It is known that $W^{1,2}(X)$ is a Banach space, but not necessarily a Hilbert space. We say that $\ms$ is an infinitesimally Hilbertian space if $W^{1,2}$ is a Hilbert space. Obviously, Riemannian manifolds (with or without boundary) are infinitesimally Hilbertian spaces.

On an infinitesimally Hilbertian space $M$, we have a natural pointwise bilinear map defined by
 \[
[W^{1,2}(M)]^2 \ni (f, g) \mapsto \la \nabla f, \nabla  g \ra:= \frac14 \Big{(}|\D (f+g)|^2-|\D (f-g)|^2\Big{)}.
\]Then we can define the Laplacian  by duality.

\begin{definition}
[Measure valued Laplacian, \cite{G-O, G-N}]
The space ${\rm D}({\bf \Delta}) \subset  W ^{1,2}(M)$ is the space of $f \in  W ^{1,2} (M)$ such that there is a measure ${\bf \mu}$ satisfying
\[
\int h \,\d{\bf \mu}= -\int \la \nabla h, \nabla f \ra \, \d  \mm, \forall h: M \mapsto  \R, ~~ \text{Lipschitz with bounded support}.
\]
In this case the measure $\mu$ is unique and we shall denote it by ${\bf \Delta} f$. If ${\bf \Delta} f \ll m$, we denote its density by $\Delta f$.
\end{definition}

We have the following proposition characterizing the curvature-dimensions conditions $\rcd$ and $\rcdkn$ through non-smooth Bakry-\'Emery theory. We say that a space is  $\rcd$/$\rcdkn$ if it is a  ${\rm CD}(K, \infty)$/${\rm CD}^*(K, N)$  space which are defined by  Lott-Sturm-Villani in \cite{Lott-Villani09, S-O1, S-O2} and 
Bacher-Sturm in \cite{BS-L}, equipped with an infinitesimally Hilbertian Sobolev space. For more details, see \cite{AGS-M} and \cite{AGMR-R}.

We define  ${\rm TestF}(M) \subset W^{1,2}(M)$, the set of test functions  by
\[
{\rm TestF}(M):= \Big\{f \in {\rm D} ({\bf \Delta}) \cap L^\infty: |\D f|\in L^\infty~~ {\rm and}~~~ \Delta f  \in W^{1,2}(M)\cap L^\infty(M) \Big \}.
\] It is known that ${\rm TestF}(M)$ is dense in $W^{1,2}(M)$ when $M$ is $\rcd$.

\bigskip

 Let $f,g \in {\rm TestF}(M)$. We know (see \cite{S-S}) that  the measure ${\bf \Gamma}_2(f,g)$ is well-defined by
\[
{\bf \Gamma}_2(f,g)=\frac12 {\bf \Delta} \la \nabla f, \nabla g \ra -\frac12 \big{(}\la \nabla f, \nabla \Delta g \ra+\la \nabla g, \nabla \Delta f \ra\big{)}\, \mm,
\]
and we put ${\bf \Gamma}_2(f):={\bf \Gamma}_2(f,f)$. 
Then we have the following Bochner inequality on metric measure space, which can be regarded as  variant definitions of $\rcd$ and $\rcdkn$ conditions.

\begin{proposition} [Bakry-\'Emery condition, \cite{AGS-B, AGS-M}, \cite{EKS-O}]\label{becondition}
Let $M=\ms$ be  an infinitesimally Hilbertian space satisfying Sobolev-to-Lipschitz property (see \cite{AGS-B} or \cite{GH-S}  for the definition). Then it is a $\rcdkn$ space with $K \in \R$ and $N \in [1, \infty]$ if and only if
\[
{\bf \Gamma}_2(f) \geq \Big {(} K |\D f|^2+ \frac1N (\Delta f)^2 \Big{)}\,\mm
\]
for any $f \in {\rm TestF}(M)$.
\end{proposition}

Let $f \in {\rm TestF}(M)$. We define the Hessian $\H_f: \{ \nabla g: g \in {\rm TestF}(M)\}^2 \mapsto L^0(M)$ by
\[
2\H_f(\nabla g,\nabla h)=\la \nabla g, \nabla \la \nabla f, \nabla h \ra \ra +\la \nabla h, \nabla \la \nabla f, \nabla g \ra \ra-\la \nabla f, \nabla \la \nabla g, \nabla h \ra \ra
\]
 for any $g, h \in {\rm TestF}(M)$. Using the estimate obtained in \cite{S-S}, it can be seen that $\H_f$ can be extended to a symmetric $L^\infty(M)$-bilinear map on $L^2(TM)$ (see \cite{G-N} for the definition) and continuous with values in $L^0(M)$,  see Theorem 3.3.8 in \cite{G-N} for a proof. On Riemannian manifolds (with boundary), it can be seen that $\H_f$ coincides with the usual Hessian $\nabla^2 f$, $\mm$-a.e., and the Hilbert-Schimidt norms are also identified. 
 
Furthermore,  we have the following proposition.
 
 \begin{proposition}[See \cite{G-N}] \label{prop-ricci}
Let $M$ be an infinitesimally Hilbertian space satisfying Sobolev-to-Lipschitz property. Then
$M$ is $\rcd$  if and only if
\[
\bRic(\nabla f, \nabla f) \geq K|\D f|^2\,\mm
\]
for any $f\in {\rm TestF}(M)$,
where
\begin{eqnarray*}
\bRic(\nabla f, \nabla f):={\bf \Gamma}_2(f) - |\H_f|^2_{\rm HS} \, \mm.
\end{eqnarray*}
\end{proposition}

\bigskip

Now we introduce our main theorem.

\begin{theorem}[Measure-valued Ricci tensor]\label{th:ricci}
Let $M=(X, {\rm g}, e^{-V}{\rm Vol}_{\rm g})$ be a $n$-dimensional weighted  Riemannian manifold and $\Omega \subset M$ be a  submanifold with $(n-1)$-dimensional smooth  orientable  boundary. Then  the measure valued Ricci tensor on  $(\Omega, \d_\Omega, e^{-V}{\rm Vol}_{\rm g})$ 
can be computed as
\begin{equation}
\bRic_\Omega(\nabla g, \nabla g)=\Ric_V(\nabla g, \nabla g)\,e^{-V}\d {\rm Vol}_{\rm g}+II(\nabla g, \nabla g) \,e^{-V}\d \mathcal{H}^{n-1}\restr{\partial \Omega}
\end{equation}
for any  $g\in C_c^\infty$ with $ {\rm g} (N, \nabla g)=0$, where $N$ is the outwards normal vector field on $\partial \Omega$, and $\Ric_V$ is the usual Bakry-\'Emery Ricci tensor on $M$.
\end{theorem}
\begin{proof}
By integration by part formula (or Green's formula) on Riemannian manifold, we know
\[
\int {\rm g}( \nabla f, \nabla g )\,e^{-V}\d {\rm Vol}_{\rm g}=-\int f \Delta_V g\,e^{-V}\d {\rm Vol}_{\rm g}+\int_{\partial \Omega} f{\rm g}( N, \nabla g)\,e^{-V}\d \mathcal{H}^{n-1}\restr{\partial \Omega}
\]
for any $f, g \in C_c^\infty$, where $\Delta_V:=(\Delta-\nabla V)$ and  $N$ is the outwards normal vector field,  $\mathcal{H}^{n-1}\restr{\partial \Omega}$ is the $(n-1)$-dimensional Hausdorff measure on $\partial \Omega$.  From the discussions before we know
\[
\int \la \nabla f, \nabla g \ra_\Omega \,e^{-V}\d {\rm Vol}_{\rm g}=-\int f \Delta_V g\,e^{-V}\d {\rm Vol}_{\rm g}+\int_{\partial \Omega} f{\rm g}( N, \nabla g)\,e^{-V}\d \mathcal{H}^{n-1}\restr{\partial \Omega}.
\]

Therefore we  know $g \in {\rm D}({\bf \Delta}_\Omega) $ and we obtain the following formula concerning the measure-valued Laplacian
\[
{\bf \Delta}_\Omega g=\Delta_V g\,e^{-V}\d{\rm Vol}_{\rm g}- {\rm g}( N, \nabla g) \,e^{-V}\d \mathcal{H}^{n-1}\restr{\partial \Omega}.
\] 
 Therefore for any $g \in C_c^\infty$ with ${\rm g}(  N, \nabla g)=0$ on $\partial \Omega$, we know $g \in {\rm TestF}(\Omega)$.

Now we can compute the measure-valued Bakry-\'Emery tensor. Let  $g \in C_c^\infty$  with  ${\rm g}(  N, \nabla g)=0$ on $\partial \Omega$. We have
\begin{eqnarray*}
\bRic_\Omega(\nabla g, \nabla g)&=&\frac12 {\bf \Delta}_\Omega |\D g|_\Omega^2-\la \nabla g, \nabla \Delta_\Omega g\ra_\Omega \,e^{-V}\d{\rm Vol}_{\rm g}-\|\Hess_g\|^2_{\rm HS}\,e^{-V}\d{\rm Vol}_{\rm g}\\
&=&\frac12 { \Delta}_V |\nabla g|^2\,e^{-V}\d{\rm Vol}_{\rm g}-{\rm g}(\nabla g, \nabla \Delta_V g) \,e^{-V}\d{\rm Vol}_{\rm g}-\|\Hess_g\|^2_{\rm HS}\,e^{-V}\d{\rm Vol}_{\rm g}\\ 
&& -\frac12{\rm g}( N, \nabla |\nabla g|^2) \,e^{-V}\d \mathcal{H}^{n-1}\restr{\partial \Omega}\\
&=&  \Ric(\nabla g, \nabla g)\,e^{-V}\d {\rm Vol}_{\rm g}+\H_V(\nabla g, \nabla g)\,e^{-V}{\rm Vol}_{\rm g}\\
&&~~~~~~~~~~~~~~~-\frac12{\rm g}( N, \nabla |\nabla g|^2) \,e^{-V}\d \mathcal{H}^{n-1}\restr{\partial \Omega}\\
&=& \Ric_V(\nabla g, \nabla g)\,e^{-V}\d {\rm Vol}_{\rm g}-\frac12{\rm g}( N, \nabla |\nabla g|^2) \,e^{-V}\d \mathcal{H}^{n-1}\restr{\partial \Omega},
\end{eqnarray*}
where we use Bochner formula at the third equality and $\Ric_V=\Ric+\H_V$ is the Bakry-\'Emery Ricci tensor on weighted Riemannian manifold w.r.t the weight $e^{-V}$.

By definition of second fundamental form, we have
\[
II(\nabla g, \nabla g)={\rm g}( \nabla_{\nabla g} N, \nabla g)={\rm g}\big ( \nabla {\rm g}(N, \nabla g ), \nabla g \big )-\frac12 {\rm g}( N, \nabla |\nabla g|^2) .
\]
However, we assume that  ${\rm g}( N, \nabla g)=0$ on $\partial \Omega$. Hence ${\rm g}( \nabla_{\nabla g} N, \nabla g)=-\frac12 {\rm g}( N, \nabla |\nabla g|^2)$.

Finally, we obtain
\begin{equation}
\bRic_\Omega(\nabla g, \nabla g)=\Ric_V(\nabla g, \nabla g)\,\d {\rm Vol}_{\rm g}+II (\nabla g, \nabla g)\,e^{-V}\d \mathcal{H}_{n-1}\restr{\partial \Omega}
\end{equation}
for any  $g\in C_c^\infty$ with ${\rm g}(N, \nabla g)=0$.
\end{proof}

\bigskip

In the next corollary we will see that the space   $\{g: g\in C_c^\infty, {\rm g}(N, \nabla g)=0\}\subset {\rm TestF}(\Omega)$ is big enough to characterize the Ricci curvature and the mean curvature.

\begin{corollary}[Rigidity: convexity of the boundary]
Let   $(\Omega, \d_\Omega, e^{-V}{\rm Vol}_{\rm g})$ be a space as in Theorem \ref{th:ricci}. Then it  is $\rcd$   if and only if $\partial \Omega$ is convex and $\Ric_V \geq K$ on $\Omega$. 
\end{corollary}
\begin{proof}
If $\Omega$ is $\rcd$, then from Proposition \ref{prop-ricci} we know  $\bRic_\Omega(\nabla g, \nabla g) \geq K|\nabla g|^2{\rm Vol}_{\rm g}$ for any $g\in {\rm TestF}(\Omega)$. By Theorem \ref{th:ricci} we know $\Ric_V(\nabla g, 
\nabla g) \geq K|\nabla g|^2$ and $II (\nabla g, 
\nabla g) \geq 0$ for any  $g\in C_c^\infty$ with ${\rm g}(N, \nabla g)=0$.

On one hand,  for any  $g\in C_c^\infty(\Omega)$ with support inside $\Omega$, we know $g \in {\rm TestF}$. Applying Theorem \ref{th:ricci} with  any of these $g$,   we know  $\Ric_V(\nabla g, \nabla g) \geq K|\nabla g|^2$, hence $\Ric_V \geq K$. 
On the other hand, for any $g\in C_c^\infty(\partial \Omega)$. By Cauchy–Kovalevskaya theorem we know the Cauchy problem:
\begin{itemize}
\item [1)] $f=g$ on $\partial \Omega$,
\item [2)] ${\rm g}(\nabla f, N)=0$ on $\partial \Omega$
\end{itemize}
has a local analytical solution $\bar{g}$. Furthermore, by multiplying an appropriate smooth cut-off function we can assume further that  $ {\bar g}\in C_c^\infty(\Omega)$  and $\bar g \in {\rm TestF}(\Omega)$. Applying Theorem \ref{th:ricci} with   $\bar g$ ,  we know $II(\nabla  g, \nabla  g) \geq 0$. Since $g$ is arbitrary, we know $II \geq 0$.

Conversely, if $\partial \Omega$ is convex we know $\Omega$ is locally  convex in the ambient space $X$ (see e.g. \cite{Bishop-I}). Combining with $\Ric_V \geq 0$ we know  $\Omega$ is locally $\rcd$.  By local to global property  of $\rcd$ condition (see e.g. \cite{S-O1}), we prove the result.
\end{proof}

\begin{remark}
In this corollary,  we only study the manifolds with boundary which can be regarded as a submanifold with orientable boundary. Since the problem we are considering is local, it is not more restrictive than general case.
\end{remark}

\begin{remark}
In \cite{AGS-C} Ambrosio-Gigli-Savar\'e identify the gradient flow of Boltzman entropy with the (Neumann) heat flow. In \cite{AGS-M} they prove the exponential contraction of heat flows in Wasserstein distance. Combining the result of Wang (see Theorem 3.3.2 in \cite{WFY-A}) we can also prove this result. 
\end{remark}

\begin{corollary}
A  $N$-dimensional  Riemannian manifold with  boundary is $\rcd$ if and only if it is $\rcdkn$.
\end{corollary}

The next corollary characterize  the Ricci-flat space as a metric measure space.

\begin{corollary}
Let $M$ and $\Omega$ be as above. Then $\Omega$ is a Ricci flat space, i.e. $\bRic_\Omega=0$,   if and only if it is a minimal hypersurface with zero Ricci curvature inside.
\end{corollary}

\def\cprime{$'$}

\end{document}